\documentclass[11pt]{article}
\usepackage[hscale=.7, vscale=.75]{geometry}
\usepackage[noblocks]{authblk}
\usepackage{xspace,amsmath,amsfonts,amssymb,bbm,ifthen,mathtools}
\usepackage[standard, thmmarks, amsmath]{ntheorem}
\usepackage[heavycircles]{stmaryrd}
\usepackage{bbold}
\usepackage[sort, noadjust]{cite}
\usepackage{wasysym}
\usepackage[hidelinks]{hyperref}
\usepackage{breakurl}

\newcommand*\qedhere{}

\usepackage{tikz}
\usetikzlibrary{arrows}
\tikzset{->, auto, >=latex', label distance=-.25cm, shorten <=-3pt,
  shorten >=-3pt}
\tikzset{every node={font=\footnotesize}}

\usepackage[theme=default-plain, charsperline=80]{jlcode}
\let\oldjlinl\jlinl
\def\jlinl#1{\textup{\oldjlinl{#1}}}

\usepackage{float}
\floatstyle{plaintop}
\newfloat{lstfloat}{htbp}{lop}
\floatname{lstfloat}{Listing}

\newcommand*\glue{\mathbin{\triangleright}}
\newcommand*\N{{\textup{\textsf{N}}}}
\newcommand*\twotwo{{\textup{\textsf{2+2}}}}
\newcommand*\intpt{\tiny {\ensuremath \Circle}}
\newcommand*\liff{\Leftrightarrow}
\newcommand*\PP{\textup{\textsf{P}}}
\newcommand*\SP{\textup{\textsf{SP}}}
\newcommand*\GP{\textup{\textsf{GP}}}
\newcommand*\IP{\textup{\textsf{IP}}}
\newcommand*\ICI{\textup{\textsf{ICI}}}
\newcommand*\GPI{\textup{\textsf{GPI}}}
\newcommand*\cat[1]{\text{\textup{\textsf{#1}}}}
\newcommand*\iPos{\cat{iPos}}
\newcommand*\id{\textup{\textsf{id}}}
\newcommand*\src{\textit{src}}
\newcommand*\tgt{\textit{tgt}}
\newcommand*\Nat{\mathbbm{N}}
\newcommand*\mcal[1]{\mathcal{#1}}
\newcommand*\from{\leftarrow}
\newcommand*\bigmid{\mathrel{\big|}}
\newcommand*\indeg{\textup{\textsf{ideg}}}
\newcommand*\outdeg{\textup{\textsf{odeg}}}
\newcommand*\inhash{\textup{\textsf{ihash}}}
\newcommand*\outhash{\textup{\textsf{ohash}}}
\newcommand*\NN{{\textup{\textsf{N\hspace{-.65ex}N}}}}
\newcommand*\NPLUS{{\textup{\textsf{M}}}}
\newcommand*\NMINUS{{\textup{\textsf{W}}}}
\newcommand*\TC{{\textup{\textsf{3C}}}}
\newcommand*\LN{{\textup{\textsf{LN}}}}
\newcommand*\cf{\textit{cf.}\xspace}

\DeclareRobustCommand{\single}[1]{%
  \ifthenelse{\equal{#1}{4}}{%
    \mcal{S}%
  }{%
    \ifthenelse{\equal{#1}{1}}{%
      \mcal T%
    }{%
      \errorUndefinedArgument%
    }}}

\title{Generating Posets with Interfaces}

\author{Olavi \"{A}ik\"{a}s}%
\affil{{\'E}cole polytechnique, Palaiseau, France\thanks{Bachelor
    internship at École polytechnique in 2021; no current
    affiliation}}%
\author{Uli Fahrenberg}%
\affil{EPITA Research and Development Laboratory (LRDE),
  France\thanks{Most of this work conducted while author employed at
    École polytechnique}}%
\author{Christian Johansen}%
\affil{Norwegian University of Science and Technology (NTNU), Norway}%
\author{Krzysztof Ziemia\'{n}ski}%
\affil{University of Warsaw, Poland}%

\begin{document}

\maketitle

\begin{abstract}
  We generate and count isomorphism classes of gluing-parallel posets
  with interfaces (iposets) on up to eight points, and on up to ten
  points with interfaces removed.  In order to do so, we introduce a
  new class of iposets with full interfaces and show that considering
  these is sufficient.  We also describe the software (written in
  Julia) that we have used for our exploration and define a new
  incomplete isomorphism invariant which may be computed in polynomial
  time yet identifies only very few pairs of non-isomorphic iposets.
\end{abstract}

\section{Introduction}

In concurrency theory, partially ordered sets (posets) are used to
model executions of programs which exhibit both sequentiality and
concurrency of events \cite{DBLP:journals/ipl/Winkowski77,
  Pratt86pomsets, DBLP:books/sp/Vogler92}.  Series-parallel posets
have been investigated due to their algebraic malleability---they are
freely generated by serial and parallel composition
\cite{DBLP:journals/fuin/Grabowski81,
  DBLP:journals/tcs/Gischer88}---and form a model of concurrent Kleene
algebra \cite{DBLP:journals/jlp/HoareMSW11}.  Interval orders are
another class of posets that arise naturally in the semantics of Petri
nets \cite{DBLP:books/sp/Vogler92, DBLP:journals/tcs/JanickiK93,
  DBLP:journals/iandc/JanickiY17}, higher-dimensional automata
\cite{DBLP:journals/tcs/Glabbeek06, Hdalang}, and distributed systems
\cite{DBLP:journals/cacm/Lamport78}.  Series-parallel posets and
interval orders are incomparable.

This paper continues work begun in
\cite{DBLP:conf/RelMiCS/FahrenbergJST20} to consolidate
series-parallel posets and interval orders.  To this end, we have
equipped posets with interfaces and extended the serial composition to
an operation which glues posets along their interfaces.  We have
investigated the algebraic structure of the so-defined gluing-parallel
posets, which encompass both series-parallel posets and interval
orders, in \cite{DBLP:conf/RelMiCS/FahrenbergJST20, BeyondN2}.  Here
we concern ourselves with the combinatorial properties of this class.

An iposet is a poset with interfaces.  We generate (and count) all
isomorphism classes of iposets and of gluing-parallel iposets on up to
8 points, and of gluing-parallel posets (with interfaces removed) on
up to 10 points.  In order to do so, we introduce a new subclass of
iposets with full interfaces and then generate all isomorphism classes
of such ``\!\emph{Winkowski}'' iposets on up to 8 points and of
gluing-parallel Winkowski iposets on up to 9 points.

We have found eleven forbidden substructures for gluing-parallel
(i)posets, five on 6 points, one on 8 points, and five other on 10
points.  We currently do not know whether there are any forbidden
substructures on 11 points or more.

To conduct our exploration we have written software in Julia, using
the LightGraphs package.  We use a recursive algorithm to generate
iposets and Julia's built-in threading support for parallelization.
For isomorphism checking we use a new incomplete invariant which may
be computed in linear time yet identifies only relatively few pairs of
non-isomorphic (i)posets.  We also detail the software and process
used to find forbidden substructures.  Our software and generated data
are freely available; McKay's similarly freely available data has been
of great help in our work.

After a preliminary Section \ref{se:posets} on posets we introduce
interfaces in Section \ref{se:iposets}.  Section \ref{se:software}
then reports on our software and Section \ref{se:forbidden} on
forbidden substructures.  Before we then can examine Winkowski iposets
in Section \ref{se:wink} we need to concern ourselves with discrete
iposets in Section \ref{se:discrete}.

We expose the numbers of non-isomorphic (i)posets in various classes
throughout the paper.  Table \ref{ta:spio} shows the numbers of posets,
series-parallel posets, interval orders, the union of the latter two
classes, and series-parallel interval orders.  Table \ref{ta:gpi}
counts iposets and gluing-parallel (i)posets, Table \ref{ta:discrete}
shows the numbers of some subclasses of discrete iposets, and Table
\ref{ta:wink} exposes the numbers of Winkowski and gluing-parallel
Winkowski iposets.  The appendix contains the counts of iposets and
gluing-parallel iposets, and of their Winkowski subclasses, split by
the numbers of sources and targets.

The main contributions of this paper are, especially when compared to
\cite{BeyondN2}, the exposition of the new subclass of Winkowski
iposets and the showcase of Julia as a programming language for
combinatorial exploration.  Further, we believe that our new
incomplete isomorphism invariant may be useful also in other contexts,
but this remains to be explored.

\section{Posets}
\label{se:posets}

A poset $(P,\mathord<)$ is a finite set $P$ equipped with an
irreflexive transitive binary relation $<$ (asymmetry of $<$ follows).
We use Hasse diagrams to visualize posets, but put greater elements to
the right of smaller ones.  Posets are equipped with a serial and a
parallel composition.  They are based on the disjoint union
(coproduct) of sets, which we write
$X\sqcup Y = \{(x,1)\mid x\in X\}\cup \{(y,2)\mid y \in Y\}$.

\begin{definition}
  Let $(P_1, <_1)$ and $(P_2, <_2)$ be posets.
  \begin{enumerate}
  \item The \emph{parallel composition} $P_1\otimes P_2$ is the
    coproduct with $P_1\sqcup P_2$ as carrier set and order defined as
    \begin{equation*}
      (p,i)<(q,j) \,\Leftrightarrow\, i=j \land p<_i q,\qquad
      i,j\in\{1,2\}.
    \end{equation*}
  \item The \emph{serial composition} $P_1\glue P_2$ is the ordinal
    sum, which again has the disjoint union as carrier set, but order
    defined as
    \begin{equation*}
      (p,i)<(q,j) \Leftrightarrow (i=j \land p<_i q) \lor i<j, \qquad
      i,j\in\{1,2\}.
    \end{equation*}
  \end{enumerate}
\end{definition}

A poset is \emph{series-parallel} (an \emph{sp-poset}) if it is either
empty or can be obtained from the singleton poset by finitely many
serial and parallel compositions.  It is well known
\cite{DBLP:journals/siamcomp/ValdesTL82,
  DBLP:journals/fuin/Grabowski81} that a poset is series-parallel if
and only if it does not contain the induced subposet
\begin{equation*}
  \N=\! \vcenter{\hbox{%
      \begin{tikzpicture}[y=.5cm]
        \node (0) at (0,0) {\intpt};
        \node (1) at (0,-1) {\intpt};
        \node (2) at (1,0) {\intpt};
        \node (3) at (1,-1) {\intpt};
        \path (0) edge (2) (1) edge (2) (1) edge (3);
      \end{tikzpicture}}}.
\end{equation*}
Further, generation of sp-posets is free: they form the free algebra
in the variety of double monoids.

An \emph{interval order} \cite{journals/mpsy/Fishburn70} is a
relational structure $(P,<)$ with $<$ irreflexive such that $w< y$ and
$x< z$ imply $w< z$ or $x< y$, for all $w,x,y,z\in P$.  Transitivity
of $<$ follows, and interval orders are therefore posets.  Interval
orders are precisely those posets that do not contain the induced
subposet
\begin{equation*}
  \twotwo=\! \vcenter{\hbox{%
      \begin{tikzpicture}[y=.5cm]
        \node (0) at (0,0) {\intpt};
        \node (1) at (0,-1) {\intpt};
        \node (2) at (1,0) {\intpt};
        \node (3) at (1,-1) {\intpt};
        \path (0) edge (2) (1) edge (3);
      \end{tikzpicture}}}.
\end{equation*}

Concurrency theory employs both sp-posets and interval orders (and
their labeled variants), the former for their algebraic malleability
and the latter because the precedence order of events in distributed
systems typically is an interval order.  We are interested in classes of posets
which retain the pleasurable algebraic properties of sp-posets but
include interval orders.

Posets $(P_1,\mathord<_1)$ and $(P_2,\mathord<_2)$ are isomorphic if
there is a bijection $f:P_1\to P_2$ such that for all $x,y\in P_1$,
$x<_1 y \liff f(x)<_2 f(y)$.  It is well-known (and easy to see, given
that posets are isomorphic if and only if their Hasse diagrams are)
that the poset isomorphism problem is just as hard as graph
isomorphism.  State of the art is Brinkmann and McKay
\cite{DBLP:journals/order/BrinkmannM02} which reports generating
isomorphism classes of posets on up to sixteen points.

Also \emph{counting} posets up to isomorphism is difficult and has
been achieved up to sixteen points.  On the other hand, both sp-posets
and interval orders admit generating functions, so counting these is
trivial.  Table \ref{ta:spio} shows the numbers of posets, sp-posets
and interval orders on $n$ points up to isomorphism for $n\le 11$, as
well as the numbers of posets which are sp-or-interval and those which
are series-parallel compositions of interval orders.

\begin{table}[tb]
  \centering
  \caption{Different types of posets on $n$ points: all posets;
    sp-posets; interval orders; sp or interval; sp-interval orders.}
  \label{ta:spio}
  \smallskip
  \begin{tabular}{r|rrrrr}
    $n$ & $\PP(n)$ & $\SP(n)$ & $\textsf{IO}(n)$ &
    $\textsf{SP\texttt+IO}(n)$ & $\textsf{SPIO}(n)$ \\\hline
    0 & 1 & 1 & 1 & 1 & 1 \\
    1 & 1 & 1 & 1 & 1 & 1 \\
    2 & 2 & 2 & 2 & 2 & 2 \\
    3 & 5 & 5 & 5 & 5 & 5 \\
    4 & 16 & 15 & 15 & 16 & 16 \\
    5 & 63 & 48 & 53 & 59 & 59 \\
    6 & 318 & 167 & 217 & 252 & 253 \\
    7 & 2045 & 602 & 1014 & 1187 & 1203 \\
    8 & 16.999 & 2256 & 5335 & 6161 & 6327 \\
    9 & 183.231 & 8660 & 31.240 & 35.038 & 36.449 \\
    10 & 2.567.284 & 33.958 & 201.608 & 218.770 & 229.660 \\
    11 & 46.749.427 & 135.292 & 1.422.074 \\[1ex]
    EIS & 112 & 3430 & 22493 & & 
  \end{tabular}
\end{table}

\section{Posets with Interfaces}
\label{se:iposets}

Let $[n]=\{1,\dotsc,n\}$ for $n\ge 1$ and $[0]=\emptyset$.  We write
$P^{\min}$ for the set of minimal and $P^{\max}$ for the set of
maximal elements of poset $P$.

\begin{definition}
  A \emph{poset with interfaces} (\emph{iposet}) is a poset $P$
  together with two injective functions
  \begin{equation*}
    [n]\overset{s}{\longrightarrow} P\overset{t}{\longleftarrow}[m]
  \end{equation*}
  such that the images $s([n])\subseteq P^{\min}$ and
  $t([m])\subseteq P^{\max}$.
\end{definition}

An iposet as above is denoted $(s,P,t):n\to m$.  We let $\iPos$ be the
set of iposets and define the identity iposets
$\id_n=(\id_{[n]}, [n], \id_{[n]}):n\to n$ for $n\ge 0$.  For
notational convenience we also define \emph{source} and \emph{target}
functions $\src, \tgt: \iPos\to \Nat$ which map $P:n\to m$ to
$\src(P)=n$ and $\tgt(P)=m$.  Any \emph{poset} $P$ is an iposet with
trivial interfaces, $\src(P)=\tgt(P)=0$.

Iposets $(s_1,P_1,t_1):n_1\to m_1$ and $(s_2,P_2,t_2):n_2\to m_2$ are
\emph{isomorphic} if there is a poset isomorphism $f:P_1\to P_2$ such
that $f\circ s_1= s_2$ and $f\circ t_1= t_2$; this implies $n_1=n_2$
and $m_1=m_2$.  The mappings $\src$ and $\tgt$ are invariant under
isomorphisms.

We extend the serial and parallel compositions to iposets.  Below,
$\phi_{ n, m}:[ n+ m]\to[ n]\otimes[ m]$ are the isomorphisms given by
\begin{equation*}
  \phi_{ n, m}( i)=
  \begin{cases}
    (i,1) &\text{if } i\le n, \\
    (i-n,2) &\text{if } i> n,
  \end{cases}
\end{equation*}
and $(P_1\sqcup P_2)_{/t_1\equiv s_2}$ denotes the quotient of the
disjoint union obtained by identifying $(t_1(k),1)$ with $(s_2(k),2)$
for every $k\in[m]$.

\begin{definition}
  Let $(s_1,P_1,t_1):n_1\to m_1$ and $(s_2,P_2 ,t_2):n_2\to m_2$ be
  iposets.
  \begin{enumerate}
  \item Their \emph{parallel composition} is the iposet
    $(s,P_1\otimes P_2,t):n_1+n_2\to m_1+m_2$ with
    $s=(s_1\otimes s_2)\circ \phi_{ n_1, n_2}$ and
    $t=(t_1\otimes t_2)\circ \phi_{ m_1, m_2}$.
  \item For $m_1=n_2$, their \emph{gluing composition} is the iposet
    $(s_1,P_1\glue P_2,t_2):n_1\to m_2$ with carrier set
    $(P_1\sqcup P_2)_{/t_1\equiv s_2}$ and order defined as
    \begin{equation*}
      (p,i)<(q,j) \Leftrightarrow (i=j \land p<_i q) \lor (i<j \land
      p\notin t_1[m_1] \land q\notin s_2[n_2]).
    \end{equation*}
  \end{enumerate}
\end{definition}

Thus $P_1\glue P_2$ is defined precisely if $\tgt(P_1)=\src(P_2)$, and
in that case, $\src(P_1\glue P_2)=\src(P_1)$ and
$\tgt(P_1\glue P_2)=\tgt(P_2)$.  Isomorphism classes of iposets form
the morphisms in a category with objects the natural numbers and
gluing as composition, or equivalently, a \emph{local partial
  $\ell r$-semigroup} \cite{DBLP:conf/RelMiCS/CalkFJSZ21, Chantale}.
For the parallel composition,
$\src(P_1\otimes P_2)= \src(P_1)\otimes \src(P_2)$ and
$\tgt(P_1\otimes P_2)=\tgt(P_1)\otimes \tgt(P_2)$, extending $\iPos$
to a \emph{partial interchange monoid} \cite{CranchDS20}.

\begin{remark}[Interchange]
  \label{re:interchange}
  The equation
  $(P_1\otimes P_2)\glue(Q_1\otimes Q_2)=(P_1\glue
  Q_1)\otimes(P_2\glue Q_2)$ does \emph{not} hold in general, not even
  up to isomorphism \cite{BeyondN2}; but see Lemma
  \ref{le:interchange} below for a special case.
\end{remark}

A composition $P=P_1\glue P_2$ or $P=P_1\otimes P_2$ is \emph{trivial}
if $P=P_1$ or $P=P_2$ as posets; see also Lemma \ref{le:trivialglue}
below. As before, we are interested in iposets
and posets which can be obtained from elementary iposets by finitely
many (nontrivial) gluing and parallel compositions.  Let
\begin{equation*}
  \single4 = \big\{ [0]\to [1]\from [0], \quad [0]\to [1]\from [1],
  \quad [1]\to [1]\from [0], \quad [1]\to [1]\from [1] \big\}
\end{equation*}
be the set of iposets on the singleton poset (the source and target
maps are uniquely determined by their type here).

\begin{definition}
  An iposet is \emph{gluing-parallel} (a \emph{gp-iposet}) if it is
  empty or can be obtained from elements of $\single4$ by finitely
  many applications of $\glue$ and $\otimes$.
\end{definition}

We are interested in generating and counting iposets and gp-iposets up
to isomorphism, but also in doing so for gluing-parallel
\emph{posets}: those posets which are gp-as-iposets.  The following
property defines a class in-between iposets and gp-iposets.

\begin{definition}
  Iposet $(s,P,t):n\to m$ is \emph{interface consistent} if
  $s^{-1}(x)<s^{-1}(y) \liff t^{-1}(x)<t^{-1}(y)$ for all
  $x,y\in s([n])\cap t([m])$.
\end{definition}

Here $<$ is the (implicit) natural ordering on $[n]$ and $[m]$.  It is
clear that gluing and parallel compositions of interface consistent
iposets are again interface consistent, hence any gp-iposet is
interface consistent.  Table \ref{ta:gpi} shows the numbers of
posets, sp-posets and gp-posets up to isomorphism, as well as of
iposets, interface consistent iposets, and gp-iposets.  In the
appendix, Tables \ref{ta:ip123split} to \ref{ta:ip8split} show the
numbers of the three classes of iposets split by the numbers of their
sources and targets.

\begin{table}[tb]
  \centering
  \caption{Different types of posets and iposets on $n$ points: all
    posets; sp-posets; gp-posets; iposets; interface consistent
    iposets; gp-iposets}
  \label{ta:gpi}
  \smallskip
  \begin{tabular}{r|rrrrrr}
    $n$ & $\PP(n)$ & $\SP(n)$ & $\GP(n)$ & $\IP(n)$ & $\ICI(n)$ &
    $\GPI(n)$ \\\hline
    0 & 1 & 1 & 1 & 1 & 1 & 1 \\
    1 & 1 & 1 & 1 & 4 & 4 & 4 \\
    2 & 2 & 2 & 2 & 17 & 16 & 16 \\
    3 & 5 & 5 & 5 & 86 & 74 & 74 \\
    4 & 16 & 15 & 16 & 532 & 420 & 419 \\
    5 & 63 & 48 & 63 & 4068 & 3030 & 2980 \\
    6 & 318 & 167 & 313 & 38.933 & 28.495 & 26.566 \\
    7 & 2045 & 602 & 1903 & 474.822 & 355.263 & 289.279 \\
    8 & 16.999 & 2256 & 13.943 & 7.558.620 & 5.937.237 & 3.726.311 \\
    9 & 183.231 & 8660 & 120.442 & \\
    10 & 2.567.284 & 33.958 & 1.206.459 & \\
    11 & 46.749.427 & 135.292 & & \\[1ex]
    EIS & 112 & 3430 & 345673 & 331158 & & 331159
  \end{tabular}
\end{table}

\section{Software}
\label{se:software}

Before we continue with our exploration, we describe the software we
have used to generate most numbers in the tables contained in this
paper and to explore the structural properties of
(i)posets.\footnote{%
  Our software is available at
  \url{https://github.com/ulifahrenberg/pomsetproject/tree/main/code/20220303/},
  and the data at
  \url{https://github.com/ulifahrenberg/pomsetproject/tree/main/data/}.}
This started as a piece of Python code written to confirm or disprove
some conjectures, but did not allow us to compute the $\GP(n)$ and
$\GPI(n)$ sequences beyond $n=6$.  During the BSc internship of the
first author of this paper, it was converted to Julia, using the
LightGraphs package \cite{Graphs21}.  This conversion, and major
improvements in candidate generation and isomorphism checking (see
below), allowed us to generate all gp-posets on $n=9$ points and all
gp-iposets on $n=7$ points.  Afterwards we managed to compute
$\GP(10)$ and $\GPI(8)$ using massively parallelized
computations.\footnote{%
  \label{fn:norway}
  We have used several ARM based servers running Linux with processor
  from High Silicon, model Hi1616, 64 cores, 256 GiB memory, 36 TiB of
  local disk, provided by the University of Oslo HPC services
  \url{https://www.uio.no/english/services/it/research/hpc/}.}

Let $G_n$ denote the set of isomorphism classes of gp-iposets on $n$
points for $n\ge 0$ and
\begin{equation*}
  G_n(k, \ell) = \big\{ P\in G_n \bigmid \src(P)=k, \tgt(P)=\ell
  \big\}
\end{equation*}
for $k,\ell\in\{0,\dotsc,n\}$.  That is, $G_n(k,\ell)$ is the set of
iposets on $n$ points with $k$ points in the starting interface and
$\ell$ in the terminating interface.  Then
$G_n=\bigcup_{k, \ell} G_n(k, \ell)$, $\GPI(n)=|G_n|$, and
$\GP(n)=|G_n(0,0)|$.  Our algorithm for generating $G_n$ is recursive
and based on the following property, where we have extended $\glue$
and $\otimes$ to sets of iposets the usual way.

\begin{lemma}
  \label{le:gp-rec}
  For all $n>1$ and $0\le k,\ell\le n$,
  \begin{equation*}
    G_n(k, \ell) =%
    \bigcup_{\substack{1\le p, q<n \\ m=p+q-n \\ 0\le m<p \\ 0\le m<q}}%
    G_p(k, m)\glue G_q(m, \ell) \cup%
    \bigcup_{\substack{p+q=n \\ p, q\ge 1 \\ k_1+k_2=k \\ \ell_1+\ell_2=\ell}}%
    G_p(k_1, \ell_1)\otimes G_q(k_2, \ell_2).
  \end{equation*}
\end{lemma}

\begin{proof}
  By definition, $R\in G_n(k, \ell)$ if and only if $R$ is a gluing or
  parallel composition of smaller gp-iposets.  If $R=P\glue Q$, then
  $P\in G_p(k, m)$ and $Q\in G_q(m, \ell)$ for some $p, q<n$, and we
  can assume $p, q\ge 1$ because otherwise the composition would be
  trivial.  The number of points of $P\glue Q$ is $p+q-m$, hence
  $m=p+q-n$, and $m<p, q$ because we can assume that both $P$ and $Q$
  have at least one non-interface point; otherwise composition would
  again be trivial.

  If $R=P\otimes Q$, then $P\in G_p(k_1, \ell_1)$ and
  $Q\in G_q(k_2, \ell_2)$ for some $p+q=n$, and we can again assume
  $p, q\ge 1$, and $k_1+k_2=k$ and $\ell_1+\ell_2=\ell$ by definition
  of $\otimes$.  We have shown that $G_n(k, \ell)$ is included in the
  expression on the right-hand side; the reverse inclusion is
  trivial.
\end{proof}

\begin{lstfloat}[tb]
  \caption{Julia function (parts) to compute $G_n(k, \ell)$.}
  \label{ls:gpclosure}
  \vspace*{-3ex}
  \begin{jllisting}[numbers=left]
function gpclosure(n, k, l, iposets, filled, locks)
    #Return memoized if exists
    if filled[k, l, n]
        return (x[1] for x in vcat(iposets[:, k, l, n]...))
    end
    #If opposites exist, return opposites
    if filled[l, k, n]
        return (reverse(x[1]) for x in vcat(iposets[:, l, k, n]...))
    end
    #Otherwise, generate recursively
    lock(locks[end, k, l, n])
    #First, the gluings
    Threads.@threads for p in 1:(n-1)
        Threads.@threads for q in 1:(n-1)
            m = p + q - n
            if m < 0 || m ≥ p || m ≥ q || k > p || l > q
                continue
            end
            for P in gpclosure(p, k, m, iposets, filled, locks)
                for Q in gpclosure(q, m, l, iposets, filled, locks)
                    R = glue(P, Q)
                    neR = ne(R) #number of edges
                    pushuptoiso!(iposets[neR, k, l, n], ip, locks[neR, k, l, n])
                end
            end
        end
    end
    #Now, the parallel compositions
    ...
    filled[k, l, n] = true
    unlock(locks[end, k, l, n])
    return (x[1] for x in vcat(iposets[:, k, l, n]...))
end
  \end{jllisting}
  \vspace*{-3ex}
\end{lstfloat}

Listing \ref{ls:gpclosure} shows part of the recursive Julia function
which implements the above algorithm.  The multi-dimensional array
\jlinl{iposets} is used for memoization and initiated with the four
singletons in $\single4$, \jlinl{filled} is used to denote which parts
of \jlinl{iposets} have been computed, and \jlinl{locks} is used for
locking.  The heart of the procedure is the call to
\jlinl{pushuptoiso!} in line 23 which checks whether \jlinl{iposets}
already contains an element isomorphic to \jlinl{ip} and, if not,
pushes it into the array.

Due to the multi-threaded implementation and tight locking, we were
able to generate $G_7$ in about 4 minutes on a standard laptop.
Generating $G_8$ took altogether 300 hours in a distributed
computation to generate each $G_8(k, \ell)$ separately on four
different computers: two standard laptops and two Norwegian
supercomputers.  For generating iposets and analyzing forbidden
substructures we have benefited greatly from Brendan McKay's poset
collections.\footnote{See
  \url{http://users.cecs.anu.edu.au/~bdm/data/digraphs.html}.}

Deciding whether iposets are isomorphic is just as difficult as for
posets.  Brinkmann and McKay \cite{DBLP:journals/order/BrinkmannM02}
develop an algorithm to compute canonical representations: mappings
$f$ from posets to labeled posets so that $f(P)=f(P')$ precisely if
$P$ and $P'$ are isomorphic.  It is clear that if any such algorithm
were to run in polynomial time, then poset isomorphism, and thus also
graph isomorphism, would be in \textsf{P}.

Canonical representations are complete isomorphism invariants.  In our
software we are instead using an \emph{incomplete} isomorphism
invariant inspired by bisimulation \cite{book/Milner89} which can be
computed in polynomial time.  For digraph $G$ and a point $x$ of $G$,
denote by $\indeg(x)$ and $\outdeg(x)$ the in- and out-degrees of $x$
in $G$.

\begin{definition}
  An \emph{in-out bisimulation} between digraphs $G_1=(V_1, E_1)$ and
  $G_2=(V_2, E_2)$ is a relation $R\subseteq V_1\times V_2$ such that
  \begin{itemize}
  \item for all $(x_1, x_2)\in R$, $\indeg(x_1)=\indeg(x_2)$ and
    $\outdeg(x_1)=\outdeg(x_2)$;
  \item for all $(x_1, x_2)\in R$ and $(x_1, y_1)\in E_1$, there
    exists $(x_2, y_2)\in E_2$ with $(y_1, y_2)\in R$;
  \item for all $(x_1, x_2)\in R$ and $(x_2, y_2)\in E_2$, there
    exists $(x_1, y_1)\in E_1$ with $(y_1, y_2)\in R$.
  \end{itemize}
\end{definition}

Note that this is the same as a standard bisimulation
\cite{book/Milner89} between the transition systems (without initial
state) given by enriching digraphs with propositions stating each
vertex's in- and out-degree.  Digraphs are said to be in-out-bisimilar
if there exists an in-out-bisimulation joining them.

\begin{definition}
  Let $P$ be a poset.  The functions $\inhash, \outhash: P\to \Nat$
  are the least fixed points to the equations
  \begin{equation*}
    \inhash(x) = \indeg(x) + |P| \sum_{y<x} \inhash(y), \qquad
    \outhash(x) = \outdeg(x) + |P| \sum_{x<y} \outhash(y).
  \end{equation*}
\end{definition}

By acyclicity these hashes are well defined, and they may be computed
in linear time.  A \emph{hash isomorphism} between posets $P$ and $Q$
is a bijection $f:P\to Q$ such that
\begin{equation*}
  (\inhash(f(x)), \outhash(f(x)))=(\inhash(x), \outhash(x))
\end{equation*}
for all $x\in P$.

\begin{samepage}
\begin{lemma}
  Let $P$ and $Q$ be posets.
  \begin{enumerate}
  \item If $P$ and $Q$ are isomorphic, then they are hash isomorphic.
  \item If $P$ and $Q$ are hash isomorphic, then they are
    in-out-bisimilar.
  \end{enumerate}
\end{lemma}
\end{samepage}

\begin{proof}
  If $f:P\to Q$ is an isomorphism, then it is also a hash isomorphism.
  If $f$ is a hash isomorphism, then the relation defined by $f$ is an
  in-out-bisimulation.
\end{proof}

Checking for existence of a hash isomorphism can be done in polynomial
time, for example by sorting the hashes.

\begin{figure}[tbp]
  \centering
  \begin{tikzpicture}[y=.7cm]
    \begin{scope}
      \node (1) at (0,0) {\intpt};
      \node (2) at (0,1) {\intpt};
      \node (0) at (0,2) {\intpt};
      \node (5) at (1,0) {\intpt};
      \node (3) at (1,1) {\intpt};
      \node (4) at (1,2) {\intpt};
      \path (0) edge (4) (0) edge (3) (2) edge (3) (1) edge (5);
    \end{scope}
    \begin{scope}[shift={(2,0)}]
      \node (1) at (0,0) {\intpt};
      \node (2) at (0,1) {\intpt};
      \node (0) at (0,2) {\intpt};
      \node (5) at (1,0) {\intpt};
      \node (3) at (1,1) {\intpt};
      \node (4) at (1,2) {\intpt};
      \path (0) edge (4) (0) edge (3) (2) edge (5) (1) edge (5);
    \end{scope}
    \begin{scope}[shift={(6,0)}]
      \node (2) at (0,0) {\intpt};
      \node (1) at (0,1) {\intpt};
      \node (0) at (0,2) {\intpt};
      \node (5) at (1,0) {\intpt};
      \node (4) at (1,1) {\intpt};
      \node (3) at (1,2) {\intpt};
      \path (0) edge (3) (0) edge (4) (1) edge (3) (1) edge (4) (2)
      edge (5);
    \end{scope}
    \begin{scope}[shift={(8,0)}]
      \node (2) at (0,0) {\intpt};
      \node (1) at (0,1) {\intpt};
      \node (0) at (0,2) {\intpt};
      \node (5) at (1,0) {\intpt};
      \node (4) at (1,1) {\intpt};
      \node (3) at (1,2) {\intpt};
      \path (0) edge (3) (0) edge (4) (1) edge (3) (1) edge (5) (2)
      edge (4);
    \end{scope}
  \end{tikzpicture}
  \caption{The two pairs of non-isomorphic posets on 6 points which
    are hash isomorphic.}
  \label{fi:hashnoiso6}
\end{figure}
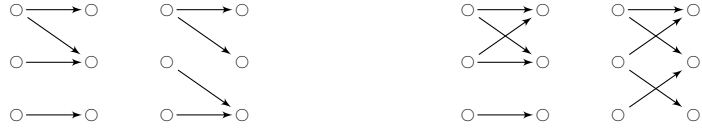

\begin{table}[bp]
  \centering
  \caption{Numbers of pairs of non-isomorphic but hash isomorphic
    posets on $n$ points; their proportion as part of all pairs of
    non-isomorphic posets; average numbers of bijections checked for
    isomorphism.}
  \label{ta:hashnoiso7+}
  \smallskip
  \begin{tabular}{r|rrr}
    $n$ & $\textsf{NIHI}(n)$ & $\textsf{NIHI}(n) / \PP(n)^2$ &
    $\textsf{nperm}(n)$ \\\hline
    5 & 0 & $0\phantom{{}\cdot 10^{-6}}$ & $0\phantom{{}\cdot
      10^{-6}}$ \\
    6 & 2 & $2\cdot 10^{-6}$ & $8\cdot 10^{-6}$ \\
    7 & 45 & $1\cdot 10^{-6}$ & $6\cdot 10^{-6}$ \\
    8 & 928 & $3\cdot 10^{-7}$ & $2\cdot 10^{-6}$ \\
    9 & 20443 & $6\cdot 10^{-8}$ & $4\cdot 10^{-7}$
  \end{tabular}
\end{table}

\begin{example}
  Hash isomorphisms are complete for posets on up to 5 points: if
  $|P|, |Q|\le 5$, then $P$ and $Q$ are isomorphic if and only if they
  are hash isomorphic.  On 6 points, there are two pairs of
  non-isomorphic posets which are hash isomorphic, depicted in Figure
  \ref{fi:hashnoiso6}.  Proportionally to the number of \emph{all}
  pairs of non-isomorphic posets, the number of ``false positives''
  grows rather slowly, see Table \ref{ta:hashnoiso7+}.  Hence it
  appears that hash isomorphism is a rather tight invariant which
  allows one to avoid most of the costly isomorphism checks.
\end{example}

\begin{lstfloat}[tb]
  \caption{Julia function to check poset isomorphism.}
  \label{ls:posetiso}
  \vspace*{-3ex}
  \begin{jllisting}[numbers=left]
function isomorphic(P::Poset, pv::Vprof, Q::Poset, qv::Vprof)
    #Start with the easy stuff
    P == Q && return true
    n = nv(P) #number of points
    (n != nv(Q) || ne(P) != ne(Q)) && return false
    #Check for hash isomorphism
    used = zeros(Bool, n)
    @inbounds for i in 1:n
        found = false
        for j in 1:n
            if pv[i] == qv[j] && !used[j]
                used[j] = true; found = true
                break
            end
        end
        !found && return false
    end
    #Collect all hash isomorphisms, mapping points to their possible images
    targets = Array{Array{Int}}(undef, n)
    targets .= [[]]
    @inbounds for v in 1:n
        for u in 1:n
            pv[v] == qv[u] && push!(targets[v], u)
        end
    end
    #Check all target permutations if they are isos
    for pos_isom in Iterators.product(targets...)
        if bijective(pos_isom) && isomorphic(P, Q, pos_isom)
            return true
        end
    end
    return false
end
  \end{jllisting}
  \vspace*{-3ex}
\end{lstfloat}

Listing \ref{ls:posetiso} shows our Julia code for checking whether
two posets are isomorphic.  The hashes are precomputed, so the function
\jlinl{isomorphic} takes two posets \jlinl{P} and \jlinl{Q} as
arguments as well as their hashes, \jlinl{pv} and \jlinl{qv} of type
\jlinl{Vprof} for ``vertex profile''.  After checking whether there is
a hash isomorphism, the hashes are used to constrain possible
isomorphisms: only bijections \jlinl{pos_isom} which are hash
isomorphisms are given to the isomorphism checker \jlinl{isomorphic(P,
  Q, pos_isom)}.  Table \ref{ta:hashnoiso7+} also shows the averages
for how many bijections are checked whether they are isomorphisms.

\section{Forbidden Substructures}
\label{se:forbidden}

Recall that an induced subposet $(Q, \mathord<_Q)$ of a poset
$(P, \mathord<_P)$ is a subset $Q\subseteq P$ with the order $<_Q$ the
restriction of $<_P$ to $Q$.  We have shown in \cite{BeyondN2} that
gluing-parallel posets are closed under induced subposets: if $Q$ is
an induced subposet of a gp-poset $P$, then also $Q$ is
gluing-parallel.  This begs the question whether gp-posets admit a
finite set of forbidden substructures: a set $\mcal F$ of posets which
are incomparable under the induced-subposet relation and such that any
poset is gluing-parallel if and only if it contains none of the
structures in $\mcal F$ as induced subposets.

\begin{proposition}[\cite{DBLP:conf/RelMiCS/FahrenbergJST20}]
  \label{pr:forbidden1}
  The following posets are contained in~$\mcal F$:
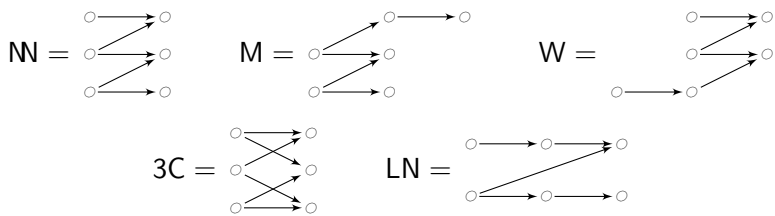
  \begin{gather*}
    \NN =\! \vcenter{\hbox{%
        \begin{tikzpicture}[y=.5cm]
          \node (0) at (0,0) {\intpt};
          \node (1) at (0,-1) {\intpt};
          \node (2) at (0,-2) {\intpt};
          \node (3) at (1,0) {\intpt};
          \node (4) at (1,-1) {\intpt};
          \node (5) at (1,-2) {\intpt};
          \path (0) edge (3) (1) edge (3) (1) edge (4) (2) edge (4)
          (2) edge (5);
        \end{tikzpicture}}}
    \qquad%
    \NPLUS =\! \vcenter{\hbox{%
        \begin{tikzpicture}[y=.5cm]
          \node (1) at (0,-1) {\intpt};
          \node (2) at (0,-2) {\intpt};
          \node (3) at (1,0) {\intpt};
          \node (4) at (1,-1) {\intpt};
          \node (5) at (1,-2) {\intpt};
          \node (6) at (2,0) {\intpt};
          \path (3) edge (6) (1) edge (3) (1) edge (4) (2) edge (4)
          (2) edge (5);
        \end{tikzpicture}}}
    \qquad%
    \NMINUS =\! \vcenter{\hbox{%
        \begin{tikzpicture}[y=.5cm]
          \node (0) at (0,0) {\intpt};
          \node (1) at (0,-1) {\intpt};
          \node (2) at (0,-2) {\intpt};
          \node (3) at (1,0) {\intpt};
          \node (4) at (1,-1) {\intpt};
          \node (-1) at (-1,-2) {\intpt};
          \path (0) edge (3) (1) edge (3) (1) edge (4) (2) edge (4)
          (-1) edge (2);
        \end{tikzpicture}}}
    \\
    \TC =\! \vcenter{\hbox{%
        \begin{tikzpicture}[y=.5cm]
          \node (0) at (0,0) {\intpt};
          \node (1) at (0,-1) {\intpt};
          \node (2) at (0,-2) {\intpt};
          \node (3) at (1,0) {\intpt};
          \node (4) at (1,-1) {\intpt};
          \node (5) at (1,-2) {\intpt};
          \path (0) edge (3) (1) edge (3) (1) edge (5) (2) edge (4)
          (2) edge (5) (0) edge (4);
        \end{tikzpicture}}}
    \qquad%
    \LN =\! \vcenter{\hbox{%
        \begin{tikzpicture}[y=.7cm]
          \node (0) at (0,0) {\intpt};
          \node (1) at (0,-1) {\intpt};
          \node (2) at (1,0) {\intpt};
          \node (3) at (1,-1) {\intpt};
          \node (4) at (2,0) {\intpt};
          \node (5) at (2,-1) {\intpt};
          \path (0) edge (2) (2) edge (4) (1) edge (4) (1) edge (3)
          (3) edge (5);
        \end{tikzpicture}}}
  \end{gather*}
\end{proposition}

\begin{lstfloat}[tb]
  \caption{Julia code to find forbidden substructures.}
  \label{ls:forbiddensubs}
  \vspace*{-3ex}
  \begin{jllisting}[numbers=left]
function findforbiddensubs()
    n = 5
    fsubs = Array{Poset}(undef, 0)
    while true
        n += 1
        pngs = posetsnotgp(n)
        newfsubs = nosubs(pngs, fsubs)
        if !isempty(newfsubs)
            println("Found new forbidden substructure(s) on $n points:")
            for s in newfsubs
                println(string(s))
            end
            append!(fsubs, newfsubs)
        end
    end
    return fsubs
end
function posetsnotgp(n)
    ps = posets(n)
    gps = [ip.poset for ip in gpiposets(n, 0, 0)]
    return diffuptoiso(ps, gps)
end
function nosubs(posets, subs)
    res = Array{Poset}(undef, 0)
    for p in posets
        hasnosub = true
        for s in subs
            sg, _ = subgraph(p, s)
            if sg
                hasnosub = false
                break
            end
        end
        hasnosub && push!(res, p)
    end
    return res
end
  \end{jllisting}
  \vspace*{-3ex}
\end{lstfloat}

Listing \ref{ls:forbiddensubs} shows our implementation of the
semi-algorithm to find forbidden substructures.  These are collected
in the array \jlinl{fsubs} and printed out as they are found.  The
function \jlinl{posetsnotgp} returns the posets on $n$ points which are
not gluing-parallel, using the function \jlinl{diffuptoiso} (not
shown) which computes the difference between two (i)poset arrays up to
isomorphism.  The function \jlinl{nosubs} returns all elements of
\jlinl{posets} which have no induced subposet isomorphic to any
element of \jlinl{subs}; this latter check is carried out in
\jlinl{subgraph} which we also do not show here.  Using McKay's files
of posets and our own precomputed files of iposets,
\jlinl{findforbiddensubs} finds the forbidden substructures of
Proposition \ref{pr:forbidden1} almost immediately.  After a few
seconds it finds another forbidden substructure on 8 points (see
below), and after an hour it verifies that there are no new forbidden
substructures on 9 points.

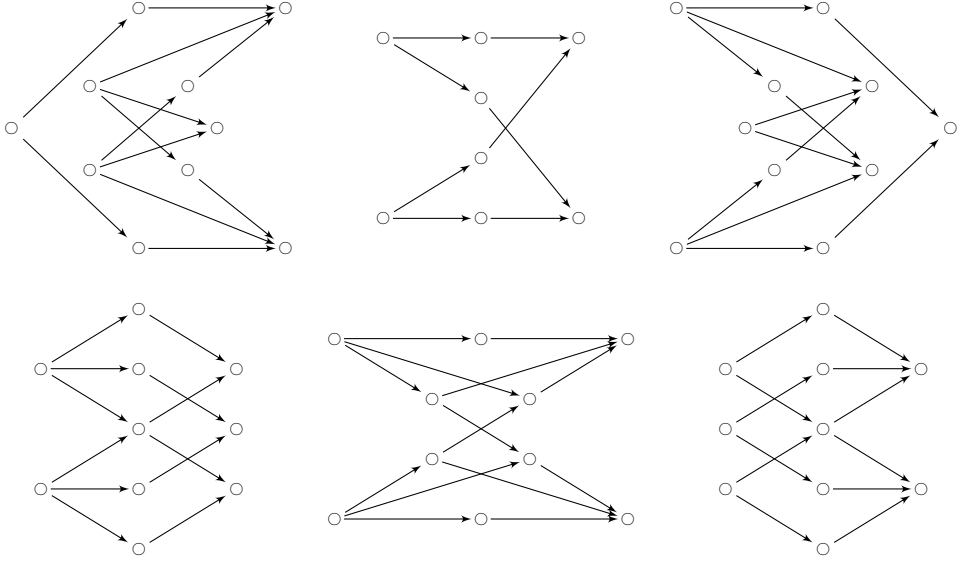
\begin{figure}[tbp]
  \centering
  \begin{tikzpicture}[x=1.3cm, y=.8cm]
    \begin{scope}[shift={(0,0)}]
      \node (2) at (0,-1.3) {\intpt};
      \node (0) at (-.8,-2) {\intpt};
      \node (1) at (0,-2.7) {\intpt};
      \node (4) at (.5,0)  {\intpt};
      \node (6) at (1,-1.3) {\intpt};
      \node (3) at (1.3,-2) {\intpt};
      \node (7) at (1,-2.7) {\intpt};
      \node (5) at (.5,-4) {\intpt};
      \node (8) at (2,0)  {\intpt};
      \node (9) at (2,-4) {\intpt};
      \path (0) edge (4) (0) edge (5) (1) edge (3) (1) edge (6) (1) edge
      (9) (2) edge (3) (2) edge (7) (2) edge (8) (4) edge (8) (5) edge
      (9) (6) edge (8) (7) edge (9);
    \end{scope}
    \begin{scope}[shift={(3,-.5)}]
      \node (0) at (0,0)  {\intpt};
      \node (1) at (0,-3) {\intpt};
      \node (2) at (1,0)  {\intpt};
      \node (3) at (1,-1) {\intpt};
      \node (4) at (1,-2) {\intpt};
      \node (5) at (1,-3) {\intpt};
      \node (6) at (2,0)  {\intpt};
      \node (7) at (2,-3) {\intpt};
      \path (0) edge (2) (0) edge (3) (1) edge (4) (1) edge (5) (2)
      edge (6) (3) edge (7) (4) edge (6) (5) edge (7);
    \end{scope}
    \begin{scope}[shift={(6,0)}]
      \node (2) at (0,0)  {\intpt};
      \node (1) at (0,-4) {\intpt};
      \node (0) at (.7,-2) {\intpt};
      \node (5) at (1,-1.3) {\intpt};
      \node (4) at (1,-2.7) {\intpt};
      \node (6) at (1.5,0)  {\intpt};
      \node (3) at (1.5,-4) {\intpt};
      \node (7) at (2,-1.3) {\intpt};
      \node (8) at (2,-2.7) {\intpt};
      \node (9) at (2.8,-2) {\intpt};
      \path (0) edge (7) (0) edge (8) (1) edge (3) (1) edge (4) (1) edge
      (8) (2) edge (5) (2) edge (6) (2) edge (7) (3) edge (9) (4) edge
      (7) (5) edge (8) (6) edge (9);
    \end{scope}
    \begin{scope}[shift={(-.5,-5)}]
      \node (0) at (0,-1) {\intpt};
      \node (1) at (0,-3) {\intpt};
      \node (2) at (1,-2) {\intpt};
      \node (3) at (1,-1) {\intpt};
      \node (4) at (1,0)  {\intpt};
      \node (5) at (1,-4) {\intpt};
      \node (6) at (1,-3) {\intpt};
      \node (7) at (2,-1) {\intpt};
      \node (8) at (2,-3) {\intpt};
      \node (9) at (2,-2) {\intpt};
      \path (0) edge (2) (0) edge (3) (0) edge (4) (1) edge (2) (1)
      edge (5) (1) edge (6) (2) edge (7) (2) edge (8) (3) edge (9) (4)
      edge (7) (5) edge (8) (6) edge (9);
    \end{scope}
    \begin{scope}[shift={(2.5,-4.5)}]
      \node (0) at (0,-1) {\intpt};
      \node (1) at (0,-4) {\intpt};
      \node (2) at (1.5,-4) {\intpt};
      \node (3) at (1.5,-1) {\intpt};
      \node (4) at (1,-2) {\intpt};
      \node (5) at (1,-3) {\intpt};
      \node (6) at (2,-3) {\intpt};
      \node (7) at (2,-2) {\intpt};
      \node (8) at (3,-4) {\intpt};
      \node (9) at (3,-1) {\intpt};
      \path (0) edge (3) (0) edge (4) (0) edge (7) (1) edge (2) (1) edge
      (5) (1) edge (6) (2) edge (8) (3) edge (9) (4) edge (6) (4) edge
      (9) (5) edge (7) (5) edge (8) (6) edge (8) (7) edge (9);
    \end{scope}
    \begin{scope}[shift={(6.5,-5)}]
      \node (0) at (0,-1) {\intpt};
      \node (1) at (0,-3) {\intpt};
      \node (2) at (0,-2) {\intpt};
      \node (3) at (1,-2) {\intpt};
      \node (4) at (1,0)  {\intpt};
      \node (5) at (1,-4) {\intpt};
      \node (6) at (1,-1) {\intpt};
      \node (7) at (1,-3) {\intpt};
      \node (8) at (2,-1) {\intpt};
      \node (9) at (2,-3) {\intpt};
      \path (0) edge (3) (0) edge (4) (1) edge (3) (1) edge (5) (2) edge
      (6) (2) edge (7) (3) edge (8) (3) edge (9) (4) edge (8) (5) edge
      (9) (6) edge (8) (7) edge (9);
    \end{scope}
  \end{tikzpicture}
  \caption{Additional forbidden substructures for gp-posets.}
  \label{fi:forbidden2}
\end{figure}

\begin{proposition}[\cite{BeyondN2}]
  \label{pr:forbidden2}
  When restricting to posets on at most 10 points, $\mcal F$ contains
  precisely the five posets of Proposition \ref{pr:forbidden1} and the
  six posets in Figure \ref{fi:forbidden2}.
\end{proposition}

In order to find the forbidden substructures on 10 points in Figure
\ref{fi:forbidden2}, we used another, distributed algorithm which took
about two weeks to run.  We generated 45 separate files containing the
gp-iposets on 10 points obtained from gluing elements of $G_n(0, k)$
and $G_m(k,0)$ for $n\in\{1, \dotsc, 9\}$ and
$m\in\{10-n, \dotsc, 9\}$ (thus $k=10-n-m$), each reduced up to
isomorphism, and one file containing all gp-iposets on 10 points
obtained as parallel compositions of smaller gp-iposets.  Then we took
\jlinl{posets10.txt}, removed posets containing one of our forbidden
substructures on 6 or 8 points, and then successively filtered it
through these 46 files, using \jlinl{diffuptoiso}.

Whether there are further forbidden substructures (on 11 points or
more), and whether $\mcal F$ is a finite set, remains open.

\section{Discrete Iposets}
\label{se:discrete}

\begin{table}[tbp]
  \centering
  \caption{Numbers of discrete iposets, gp-discrete iposets, starters,
    and gp-starters on $n$ points.  (Numbers of terminators and
    gp-terminators are the same as in the two last columns.)}
  \label{ta:discrete}
  \smallskip
  \begin{tabular}{r|rrrr}
    $n$ & $\textsf{D}(n)$ & $\textsf{GPD}(n)$ & $\textsf{S}(n)$ &
    $\textsf{GPS}(n)$ \\\hline
    0 & 1 & 1 & 1 & 1 \\
    1 & 4 & 4 & 2 & 2 \\
    2 & 13 & 12 & 5 & 4\\
    3 & 45 & 33 & 16 & 8 \\
    4 & 184 & 88 & 65 & 16 \\
    5 & 913 & 232 & 326 & 32 \\
    6 & 5428 & 609 & 1957 & 64 \\
    7 & 37.764 & 1596 & 13.700 & 128 \\
    8 & 300.969 & 4180 & 109.601 & 256 \\
    9 & 2.702.152 & 10.945 & 986.410 & 512 \\
    10 & 26.977.189 & 28.656 & 9.864.101 & 1024 \\[1ex]
    EIS & & 27941 & 522 & 79
  \end{tabular}
\end{table}

This section explores the ``fine structure'' of iposets.  An iposet
$(s, P, t):n\to m$ is \emph{discrete} if $P$ is, it is a
\emph{starter} if, additionally, $t:[m]\to P$ is bijective, and a
\emph{terminator} if $s:[n]\to P$ is bijective.  Any discrete iposet
is the gluing of a starter with a terminator and is gluing-parallel if
and only if it is interface consistent.  The following is clear.

\begin{lemma}
  \label{le:trivialglue}
  A gluing $P=P_1\glue P_2$ is trivial if and only if $P_1$ is a
  starter or $P_2$ is a terminator. \qed
\end{lemma}

The next proposition shows numbers of some classes of discrete
iposets, see also Table~\ref{ta:discrete}.

\begin{proposition}
  Let $n\ge 0$.  Up to isomorphism,
  \begin{enumerate}
  \item there are $2^n$ gp-starters and $2^n$ gp-terminators on $n$
    points;
  \item there are $\sum_{k=0}^n \frac{n!}{k!}$ starters and
    $\sum_{k=0}^n \frac{n!}{k!}$ terminators on $n$ points;
  \item there are $\sum_{s,t=0}^n \sum_{u=\max(0,s+t-n)}^{\min(s,t)}
    \binom{s}{u} \binom{t}{u}$ gp-discrete iposets on $n$ points;
  \item there are
    $\sum_{s,t=0}^n \sum_{u=\max(0,s+t-n)}^{\min(s,t)} \binom{s}{u}
    \binom{t}{u} u!$ discrete iposets on $n$ points.
  \end{enumerate}
\end{proposition}

\begin{samepage}
The third term above can be simplified by
\begin{equation*}
  \sum_{s,t=0}^n\, \sum_{u=\max(0,s+t-n)}^{\min(s,t)} \binom{s}{u}
  \binom{t}{u} = \sum_{i=0}^n \binom{n+2+i}{n-i},
\end{equation*}
using a version of Vandermonde's identity; we are not aware of any
such simplification for the last term.
\end{samepage}

\begin{proof}
  \mbox{}
  \begin{enumerate}
  \item For any $k\in\{0,\dotsc,n\}$, there are $\binom{n}{k}$
    non-isomorphic interface consistent starters on $n$ points with
    $k$ of them in the starting interface.
  \item Similarly, there are $\binom{n}{k} k!$ non-isomorphic starters
    on $n$ points with $k$ points in the starting interface when not
    requiring interface consistency.
  \item Let $s,t,u\in\{0,\dotsc,n\}$ and consider the number of
    non-isomorphic interface consistent discrete iposets on $n$ points
    with $s$ points in the starting, $t$ points in the terminating,
    and $u$ points in both interfaces, then necessarily
    $s+t-n\le u\le \min(s,t)$.  The points not in both interfaces only
    give rise to one isomorphism class, the points in the overlap may
    be chosen in $\binom{s}{u} \binom{t}{u}$ non-isomorphic ways, and
    their order is unique by interface consistency.
  \item The argument is the same as above; but the missing interface
    consistency requirement adds a factor $u!$. \qedhere
  \end{enumerate}
\end{proof}

A discrete iposet $(s,P,t):n\to n$ is a \emph{symmetry} if it is both
a starter and a terminator, that is, $s$ and $t$ are both bijective.
All points of $P$ are in the starting and terminating interfaces, but
the permutation $t^{-1}\circ s:[n]\to [n]$ is not necessarily an
identity.  It is clear that there are precisely $n!$ non-isomorphic
symmetries on $n$ points, and that any discrete iposet $P$ may be
written $P\cong \sigma\glue Q\cong R\glue \tau$ for symmetries
$\sigma$, $\tau$ and $Q$ and $R$ interface consistent.

We finish this section by a special case of the interchange property
relating parallel and gluing compositions, \cf Remark
\ref{re:interchange}.

\begin{lemma}[\cite{BeyondN2}]
  \label{le:interchange}
  Let $P_1$, $P_2$, $Q_1$, $Q_2$ be iposets such that
  $\tgt(P_1)=\src(Q_1)$ and $\tgt(P_2)=\src(Q_2)$.  Then
  $(P_1\otimes P_2)\glue(Q_1\otimes Q_2)=(P_1\glue
  Q_1)\otimes(P_2\glue Q_2)$ if and only if $P_1\glue Q_1$ or
  $P_2\glue Q_2$ is discrete.
\end{lemma}

\section{Iposets with Full Interfaces}
\label{se:wink}

We now introduce a class of iposets where \emph{all} minimal and/or
maximal points are in the interfaces; we name these after Winkowski
\cite{DBLP:journals/ipl/Winkowski77} who, to the best of our
knowledge, was the first to consider posets with interfaces, and who
only considered such full-interface iposets.

\begin{definition}
  An iposet $(s,P,t):n\to m$ is \emph{left Winkowski} if
  $s([n])= P^{\min}$, \emph{right Winkowski} if $t([m])= P^{\max}$,
  and \emph{Winkowski} if it is both left and right Winkowski.
\end{definition}

Note that starters are precisely discrete right Winkowskis,
terminators are precisely discrete left Winkowskis, and symmetries are
precisely the discrete Winkowskis.

\begin{lemma}
  \label{le:W-glue}
  Let $P= P_1\glue P_2$ nontrivially.
  \begin{itemize}
  \item $P$ is left Winkowski if and only if $P_1$ is;
  \item $P$ is right Winkowski if and only if $P_2$ is;
  \item $P$ is Winkowski if and only if $P_1$ is left Winkowski and
    $P_2$ is right Winkowski.
  \end{itemize}
\end{lemma}

\begin{proof}
  We show that $P^{\min}=P_1^{\min}$ and $P^{\max}=P_2^{\max}$;
  the lemma then follows.  By non-triviality there must be $x\in P_1$
  which is not in the target interface.  Now
  $P_1^{\min}\subseteq P^{\min}$ by definition of $\glue$, so assume
  $y\in P^{\min}\setminus P_1^{\min}$.  Then $y\notin P_1$, which
  implies $x<y$ in contradiction to $y\in P^{\min}$.  (Note that
  non-triviality of $P=P_1\glue P_2$ is necessary here: if $P_1$ is a
  starter, we may have $y\notin P_1$ but still $y\in P^{\min}$.)  The
  proof for $P^{\max}=P_2^{\max}$ is symmetric.
\end{proof}

For parallel compositions, it is clear that $P_1\otimes P_2$ is
(left/right) Winkowski if and only if $P_1$ and $P_2$ are.  Our
immediate interest in Winkowski iposets is to speed up generation of
gp-iposets by only considering gp-Winkowskis.  It is clear that
any iposet has a decomposition $P=S\glue W\glue T$ into a starter $S$,
a Winkowski $W$ and a terminator $T$; by the next lemma, this also
holds in the gluing-parallel case.

\begin{lemma}
  \label{le:decomp-SWT}
  Any gp-iposet $P$ has a decomposition $P=S\glue W\glue T$ into a
  starter $S$, a Winkowski $W$ and a terminator $T$ which are all
  gluing-parallel.
\end{lemma}

\begin{proof}
  Let $n=|P|$ be the number of points in $P$.  If $n\le 1$, then the
  claim is trivially true as all iposets on $0$ or $1$ points are
  gluing-parallel.  Let $n\ge 2$, assume that the claim is true for
  all iposets with fewer than $n$ points, and let $P$ be
  gluing-parallel.

  If $P= P_1\glue P_2$ nontrivially with $P_1$ and $P_2$
  gluing-parallel, then by the induction hypothesis,
  $P_1= S_1\glue W_1\glue T_1$ and $P_2= S_2\glue W_2\glue T_2$ for
  $S_1$, $S_2$ gp-starters, $W_1$, $W_2$ gp-Winkowski, and $T_1$,
  $T_2$ gp-terminators.  Now
  $P= P_1\glue P_2= S_1\glue( W_1\glue T_1\glue S_2\glue W_2)\glue
  T_2$, and $W_1\glue T_1\glue S_2\glue W_2$ is gluing-parallel
  because all four components are, and is Winkowski because $W_1$ and
  $W_2$ are.

  If $P= P_1\otimes P_2$ nontrivially with $P_1$ and $P_2$
  gluing-parallel, then by the induction hypothesis,
  $P_1= S_1\glue W_1\glue T_1$ and $P_2= S_2\glue W_2\glue T_2$ for
  $S_1$, $S_2$ gp-starters, $W_1$, $W_2$ gp-Winkowskis, and $T_1$,
  $T_2$ gp-terminators.  Now
  \begin{equation*}
    P= P_1\otimes P_2=( S_1\glue W_1\glue T_1)\otimes( S_2\glue
    W_2\glue T_2)=( S_1\otimes S_2)\glue( W_1\otimes W_2)\glue(
    T_1\otimes T_2)
  \end{equation*}
  by Lemma \ref{le:interchange}, $S_1\otimes S_2$ is a gp-starter,
  $W_1\otimes W_2$ is gp-Winkowski, and $T_1\otimes T_2$ is a
  gp-terminator.
\end{proof}

For generating gluing-parallel iposets it is thus sufficient to
generate gp-Winkowskis, gp-starters and gp-terminators.  The next
lemma entails that also in the recursions these are the only classes
we need to consider.

\begin{lemma}
  \label{le:gp-SWT}
  For $P$ a gluing-parallel Winkowski iposet, the following are
  exhaustive:
  \begin{enumerate}
  \item $P=\id_0$ or $P=\id_1$;
  \item $P= P_1\otimes P_2$ nontrivially for $P_1$ and $P_2$
    gp-Winkowski;
  \item \label{en:gp-SWT.glue} $P= P_1\glue P_2$ nontrivially for
    $P_1$ gp-Winkowski or a gp-terminator and $P_2$ gp-Win\-kow\-ski
    or a gp-starter.
  \end{enumerate}
\end{lemma}

\begin{proof}
  The first two cases are clear.  Otherwise, $P= P_1\glue P_2$
  nontrivially for $P_1$ and $P_2$ gluing-parallel.  By Lemma
  \ref{le:W-glue}, $P_1$ is left Winkowski and $P_2$ right Winkowski,
  and by Lemma \ref{le:decomp-SWT} we can decompose $P_1=W_1\glue T_1$
  and $P_2=S_2\glue W_2$ into gp-starters, gp-Winkowskis and
  gp-terminators.  Then $P=W_1\glue T_1\glue S_2\glue W_2$.  There are
  four cases to consider:
  \begin{enumerate}
  \item If both $W_1$ and $W_2$ are identities, then neither $T_1$ nor
    $S_2$ are (by non-triviality of $P=P_1\glue P_2$), hence
    $P=T_1\glue S_2$ nontrivially.
  \item If $W_1$ is an identity, but $W_2$ is not, then also $T_1$ is
    not an identity.  Now if $S_2$ is an identity, then $P=T_1\glue
    W_2$ nontrivially; otherwise, $T_1\glue S_2$ is Winkowski by Lemma
    \ref{le:W-glue} and $P=(S_1\glue T_2)\glue W_2$.
  \item The case of $W_2$ being an identity but not $W_1$ is
    symmetric.
  \item If neither $W_1$ nor $W_2$ are identities, but $T_1\glue S_2$
    is, then $P=W_1\glue W_2$ nontrivially.  If also $T_1\glue S_2$ is
    not an identity, then $T_1\glue S_2\glue W_2$ is Winkowski by
    Lemma \ref{le:W-glue} and $P=W_1\glue(T_1\glue S_2\glue W_2)$
    nontrivially. \qedhere
  \end{enumerate}
\end{proof}

Denoting by $G_n^\textup{W}$, $G_n^\textup{S}$ and $G_n^\textup{T}$
the subsets of $G_n$ consisting of Winkowskis, starters, respectively
terminators, we have thus shown the following refinement of Lemma
\ref{le:gp-rec}.

\begin{lemma}
  For all $n>1$ and $0\le k,\ell\le n$,
  \begin{multline*}
    G_n^\textup{W}(k, \ell) =%
    \smash[b]{\bigcup_{\substack{1\le p, q<n \\ m=p+q-n \\ 0\le m<p \\
          0\le m<q}}} \big( G_p^\textup{W}(k, m)\cup G_p^\textup{T}(k,
    m) \big) \glue \big( G_q^\textup{W}(m, \ell)\cup G_q^\textup{S}(m,
    \ell) \big)
    \\[1ex]
    \cup \bigcup_{\substack{p+q=n \\ p, q\ge 1 \\ k_1+k_2=k \\
        \ell_1+\ell_2=\ell}} G_p^\textup{W}(k_1, \ell_1)\otimes
    G_q^\textup{W}(k_2, \ell_2).
  \end{multline*}
  \qed
\end{lemma}

\begin{table}
  \centering
  \caption{Different types of (i)posets on $n$ points: posets;
    iposets; gp-iposets; Winkowski iposets; interface consistent
    Winkowskis; gp-Winkowskis}
  \label{ta:wink}
  \smallskip
  \begin{tabular}{r|rrrrrr}
    $n$ & $\PP(n)$ & $\IP(n)$ & $\GPI(n)$ & $\textsf{WIP}(n)$ &
    $\textsf{ICW}(n)$ & $\textsf{GPWI}(n)$ \\\hline
    0 & 1 & 1 & 1 & 1 & 1 & 1 \\
    1 & 1 & 4 & 4 & 1 & 1 & 1 \\
    2 & 2 & 17 & 16 & 3 & 2 & 2 \\
    3 & 5 & 86 & 74 & 13 & 8 & 8 \\
    4 & 16 & 532 & 419 & 75 & 43 & 42 \\
    5 & 63 & 4068 & 2980 & 555 & 311 & 284 \\
    6 & 318 & 38.933 & 26.566 & 5230 & 3018 & 2430 \\
    7 & 2045 & 474.822 & 289.279 & 63.343 & 39.196 & 25.417 \\
    8 & 16.999 & 7.558.620 & 3.726.311 & 1.005.871 & 682.362 & 314.859 \\
    9 & 183.231 & & & & & 4.509.670 \\[1ex]
    EIS & 112 & 331158 & 331159
  \end{tabular}
\end{table}

Note that in order to find forbidden substructures, it is not enough
to generate $G_n^\textup{W}(0,0)$ as was the case for general iposets;
indeed $G_n^\textup{W}(0,0)=\emptyset$ for $n\ge 1$, given that the
number of interfaces for Winkowski iposets is a structural property
determined by the underlying posets.  Generating $G_7^\textup{W}$ took
about 4 seconds and $G_8^\textup{W}$ ca.\ 12 minutes on a standard
laptop (compare this with the 4 minutes for $G_7$ and 300 hours for
$G_8$).  Generating $G_9^\textup{W}$ took 79 hours on one of the
machines mentioned in footnote \ref{fn:norway}.  Table \ref{ta:wink}
shows the numbers of Winkowskis and gp-Winkowskis on $n$ points up to
isomorphism, and Tables \ref{ta:wip123split} to \ref{ta:wip9split} in
the appendix show the split into sources and targets.

\newcommand{\etalchar}[1]{$^{#1}$}
\newcommand{\Afirst}[1]{#1} \newcommand{\afirst}[1]{#1}

\clearpage
\appendix
\setcounter{table}0
\renewcommand{\thetable}{A.\arabic{table}}

\begin{table}[tbp]
  \centering
  \caption{Numbers of iposets on 1, 2 and 3 points split by number of
    sources and targets.}
  \label{ta:ip123split}
  \smallskip
  \begin{tabular}[t]{r|rr}
    $\IP(1)$\!\! & 0 & 1 \\\hline
    0 & 1 & 1 \\
    1 & & 1
  \end{tabular}
  \qquad
  \begin{tabular}[t]{r|rrr}
    $\IP(2)$\!\! & 0 & 1 & 2 \\\hline
    0 & 2 & 2 & 1 \\
    1 & & 3 & 2 \\
    2 & & & 2 \\
  \end{tabular}
  \qquad
  \begin{tabular}[t]{r|rrr}
    $\ICI(2)$\!\! & 0 & 1 & 2 \\\hline
    0 & 2 & 2 & 1 \\
    1 & & 3 & 2 \\
    2 & & & 1
  \end{tabular}

  \medskip
  \begin{tabular}[t]{r|rrrr}
    $\IP(3)$\!\! & 0 & 1 & 2 & 3 \\\hline
    0 & 5 & 6 & 4 & 1 \\
    1 & & 9 & 8 & 3 \\
    2 & & & 10 & 6 \\
    3 & & & & 6 \\
  \end{tabular}
  \qquad
  \begin{tabular}[t]{r|rrrr}
    $\ICI(3)$ & 0 & 1 & 2 & 3 \\\hline
    0 & 5 & 6 & 4 & 1 \\
    1 & & 9 & 8 & 3 \\
    2 & & & 9 & 3 \\
    3 & & & & 1
  \end{tabular}
\end{table}

\begin{table}[tbp]
  \centering
  \caption{Numbers of iposets on 4 points split by number of sources
    and targets.}
  \label{ta:ip4split}
  \smallskip
\begin{tabular}[t]{r|rrrrr}
$\IP(4)$\!\! & 0 & 1 & 2 & 3 & 4 \\\hline
0 & 16 & 22 & 19 & 8 & 1 \\
1 & & 36 & 37 & 20 & 4 \\
2 & & & 48 & 36 & 12 \\
3 & & & & 42 & 24 \\
4 & & & & & 24 \\
\end{tabular}
\qquad
\begin{tabular}[t]{r|rrrrr}
$\ICI(4)$\!\! & 0 & 1 & 2 & 3 & 4 \\\hline
0 & 16 & 22 & 19 & 8 & 1 \\
1 & & 36 & 37 & 20 & 4 \\
2 & & & 46 & 30 & 6 \\
3 & & & & 19 & 4 \\
4 & & & & & 1 \\
\end{tabular}

\medskip
\begin{tabular}[t]{r|rrrrr}
$\GPI(4)$\!\! & 0 & 1 & 2 & 3 & 4 \\\hline
0 & 16 & 22 & 19 & 8 & 1 \\
1 & & 36 & 37 & 20 & 4 \\
2 & & & 45 & 30 & 6 \\
3 & & & & 19 & 4 \\
4 & & & & & 1 \\
\end{tabular}
\end{table}

\begin{table}[tbp]
  \centering
  \caption{Numbers of iposets on 5 points split by number of sources
    and targets.}
  \label{ta:ip5split}
  \smallskip
\begin{tabular}[t]{r|rrrrrr}
$\IP(5)$\!\! & 0 & 1 & 2 & 3 & 4 & 5 \\\hline
0 & 63 & 101 & 106 & 63 & 16 & 1 \\
1 & & 180 & 214 & 148 & 48 & 5 \\
2 & & & 295 & 250 & 112 & 20 \\
3 & & & & 282 & 192 & 60 \\
4 & & & & & 216 & 120 \\
5 & & & & & & 120 \\
\end{tabular}
\qquad
\begin{tabular}[t]{r|rrrrrr}
$\ICI(5)$\!\! & 0 & 1 & 2 & 3 & 4 & 5 \\\hline
0 & 63 & 101 & 106 & 63 & 16 & 1 \\
1 & & 180 & 214 & 148 & 48 & 5 \\
2 & & & 290 & 232 & 88 & 10 \\
3 & & & & 209 & 80 & 10 \\
4 & & & & & 33 & 5 \\
5 & & & & & & 1 \\
\end{tabular}

\medskip
\begin{tabular}[t]{r|rrrrrr}
$\GPI(5)$\!\! & 0 & 1 & 2 & 3 & 4 & 5 \\\hline
0 & 63 & 101 & 106 & 62 & 16 & 1 \\
1 & & 180 & 214 & 146 & 48 & 5 \\
2 & & & 281 & 220 & 88 & 10 \\
3 & & & & 198 & 80 & 10 \\
4 & & & & & 33 & 5 \\
5 & & & & & & 1 \\
\end{tabular}
\end{table}

\begin{table}[tbp]
  \centering
  \caption{Numbers of iposets on 6 points split by number of sources
    and targets.}
  \label{ta:ip6split}
  \smallskip
\begin{tabular}[t]{r|rrrrrrr}
$\IP(6)$\!\! & 0 & 1 & 2 & 3 & 4 & 5 & 6 \\\hline
0 & 318 & 576 & 720 & 552 & 217 & 32 & 1 \\
1 & & 1131 & 1536 & 1303 & 589 & 112 & 6 \\
2 & & & 2305 & 2221 & 1212 & 320 & 30 \\
3 & & & & 2549 & 1812 & 720 & 120 \\
4 & & & & & 1872 & 1200 & 360 \\
5 & & & & & & 1320 & 720 \\
6 & & & & & & & 720 \\
\end{tabular}

\medskip
\begin{tabular}[t]{r|rrrrrrr}
$\ICI(6)$\!\! & 0 & 1 & 2 & 3 & 4 & 5 & 6 \\\hline
0 & 318 & 576 & 720 & 552 & 217 & 32 & 1 \\
1 & & 1131 & 1536 & 1303 & 589 & 112 & 6 \\
2 & & & 2289 & 2155 & 1098 & 240 & 15 \\
3 & & & & 2245 & 1242 & 280 & 20 \\
4 & & & & & 690 & 170 & 15 \\
5 & & & & & & 51 & 6 \\
6 & & & & & & & 1 \\
\end{tabular}

\medskip
\begin{tabular}[t]{r|rrrrrrr}
$\GPI(6)$\!\! & 0 & 1 & 2 & 3 & 4 & 5 & 6 \\\hline
0 & 313 & 565 & 703 & 523 & 205 & 32 & 1 \\
1 & & 1104 & 1493 & 1235 & 561 & 112 & 6 \\
2 & & & 2146 & 1931 & 993 & 240 & 15 \\
3 & & & & 1911 & 1092 & 280 & 20 \\
4 & & & & & 644 & 170 & 15 \\
5 & & & & & & 51 & 6 \\
6 & & & & & & & 1 \\
\end{tabular}
\end{table}

\begin{table}[tbp]
  \centering
  \caption{Numbers of iposets on 7 points split by number of sources
    and targets.}
  \label{ta:ip7split}
  \smallskip
\begin{tabular}[t]{r|rrrrrrrr}
$\IP(7)$\!\! & 0 & 1 & 2 & 3 & 4 & 5 & 6 & 7 \\\hline
0 & 2045 & 4162 & 6026 & 5692 & 3074 & 771 & 64 & 1 \\
1 & & 8945 & 13756 & 13925 & 8210 & 2352 & 256 & 7 \\
2 & & & 22664 & 24956 & 16465 & 5654 & 864 & 42 \\
3 & & & & 30610 & 23572 & 10440 & 2400 & 210 \\
4 & & & & & 22880 & 14400 & 5280 & 840 \\
5 & & & & & & 14040 & 8640 & 2520 \\
6 & & & & & & & 9360 & 5040 \\
7 & & & & & & & & 5040 \\
\end{tabular}

\medskip
\begin{tabular}[t]{r|rrrrrrrr}
$\ICI(7)$\!\! & 0 & 1 & 2 & 3 & 4 & 5 & 6 & 7 \\\hline
0 & 2045 & 4162 & 6026 & 5692 & 3074 & 771 & 64 & 1 \\
1 & & 8945 & 13756 & 13925 & 8210 & 2352 & 256 & 7 \\
2 & & & 22601 & 24653 & 15829 & 5024 & 624 & 21 \\
3 & & & & 29054 & 20072 & 6760 & 880 & 35 \\
4 & & & & & 14489 & 4870 & 700 & 35 \\
5 & & & & & & 1777 & 312 & 21 \\
6 & & & & & & & 73 & 7 \\
7 & & & & & & & & 1 \\
\end{tabular}

\medskip
\begin{tabular}[t]{r|rrrrrrrr}
$\GPI(7)$\!\! & 0 & 1 & 2 & 3 & 4 & 5 & 6 & 7 \\\hline
0 & 1903 & 3813 & 5423 & 4878 & 2563 & 680 & 64 & 1 \\
1 & & 8056 & 12179 & 11811 & 6865 & 2110 & 256 & 7 \\
2 & & & 19129 & 19567 & 12305 & 4246 & 624 & 21 \\
3 & & & & 21295 & 14420 & 5433 & 880 & 35 \\
4 & & & & & 10439 & 4112 & 700 & 35 \\
5 & & & & & & 1647 & 312 & 21 \\
6 & & & & & & & 73 & 7 \\
7 & & & & & & & & 1 \\
\end{tabular}
\end{table}

\begin{table}[tbp]
  \centering
  \caption{Numbers of iposets on 8 points split by number of sources
    and targets.}
  \label{ta:ip8split}
  \smallskip
\begin{tabular}[t]{r|rrrrrrrrr}
$\IP(8)$\!\! & 0 & 1 & 2 & 3 & 4 & 5 & 6 & 7 & 8 \\\hline
0 & 16999 & 38280 & 63088 & 70946 & 49255 & 18152 & 2809 & 128 & 1 \\
1 & & 89699 & 154451 & 182680 & 134680 & 53651 & 9451 & 576 & 8 \\
2 & & & 279685 & 350957 & 278197 & 122505 & 25810 & 2240 & 56 \\
3 & & & & 472927 & 410905 & 207923 & 56322 & 7392 & 336 \\
4 & & & & & 406232 & 253640 & 96600 & 20160 & 1680 \\
5 & & & & & & 218200 & 126120 & 43680 & 6720 \\
6 & & & & & & & 118080 & 70560 & 20160 \\
7 & & & & & & & & 75600 & 40320 \\
8 & & & & & & & & & 40320 \\
\end{tabular}

\medskip
\begin{tabular}[t]{r|rrrrrrrrr}
$\ICI(8)$\!\! & 0 & 1 & 2 & 3 & 4 & 5 & 6 & 7 & 8 \\\hline
0 & 16999 & 38280 & 63088 & 70946 & 49255 & 18152 & 2809 & 128 & 1 \\
1 & & 89699 & 154451 & 182680 & 134680 & 53651 & 9451 & 576 & 8 \\
2 & & & 279367 & 349229 & 273877 & 116985 & 22555 & 1568 & 28 \\
3 & & & & 463000 & 384873 & 173073 & 34857 & 2576 & 56 \\
4 & & & & & 334532 & 152970 & 30605 & 2520 & 70 \\
5 & & & & & & 68080 & 14711 & 1484 & 56 \\
6 & & & & & & & 3854 & 518 & 28 \\
7 & & & & & & & & 99 & 8 \\
8 & & & & & & & & & 1 \\
\end{tabular}

\medskip
\begin{tabular}[t]{r|rrrrrrrrr}
$\GPI(8)$\!\! & 0 & 1 & 2 & 3 & 4 & 5 & 6 & 7 & 8 \\\hline
0 & 13943 & 30333 & 48089 & 50187 & 32790 & 12348 & 2251 & 128 & 1 \\
1 & & 68571 & 113701 & 125539 & 88295 & 36791 & 7789 & 576 & 8 \\
2 & & & 193330 & 221192 & 164078 & 73774 & 17438 & 1568 & 28 \\
3 & & & & 263828 & 206161 & 98655 & 25233 & 2576 & 56 \\
4 & & & & & 169476 & 85192 & 22937 & 2520 & 70 \\
5 & & & & & & 44362 & 12173 & 1484 & 56 \\
6 & & & & & & & 3559 & 518 & 28 \\
7 & & & & & & & & 99 & 8 \\
8 & & & & & & & & & 1 \\
\end{tabular}
\end{table}

\begin{table}[tbp]
  \centering
  \caption{Numbers of Winkowski iposets on 1, 2 and 3 points split by
    number of sources and targets.}
  \label{ta:wip123split}
  \smallskip
  \begin{tabular}[t]{r|rr}
    $\textsf{WIP}(1)$\!\! & 0 & 1 \\\hline
    0 & 0 & 0 \\
    1 & & 1 \\
  \end{tabular}
  \qquad
  \begin{tabular}[t]{r|rrr}
    $\textsf{WIP}(2)$\!\! & 0 & 1 & 2 \\\hline
    0 & 0 & 0 & 0 \\
    1 & & 1 & 0 \\
    2 & & & 2 \\
  \end{tabular}
  \qquad
  \begin{tabular}[t]{r|rrr}
    $\textsf{GPWI}(2)$\!\! & 0 & 1 & 2 \\\hline
    0 & 0 & 0 & 0 \\
    1 & & 1 & 0 \\
    2 & & & 1 \\
  \end{tabular}

  \medskip
  \begin{tabular}[t]{r|rrrr}
    $\textsf{WIP}(3)$\!\! & 0 & 1 & 2 & 3 \\\hline
    0 & 0 & 0 & 0 & 0 \\
    1 & & 1 & 1 & 0 \\
    2 & & & 4 & 0 \\
    3 & & & & 6 \\
  \end{tabular}
  \qquad
  \begin{tabular}[t]{r|rrrr}
    $\textsf{GPWI}(3)$\!\! & 0 & 1 & 2 & 3 \\\hline
    0 & 0 & 0 & 0 & 0 \\
    1 & & 1 & 1 & 0 \\
    2 & & & 4 & 0 \\
    3 & & & & 1 \\
  \end{tabular}
\end{table}

\begin{table}[tbp]
  \centering
  \caption{Numbers of Winkowski iposets on 4 points split by number of
    sources and targets.}
  \label{ta:wip4split}
  \begin{tabular}[t]{r|rrrrr}
    $\textsf{WIP}(4)$\!\! & 0 & 1 & 2 & 3 & 4 \\\hline
    0 & 0 & 0 & 0 & 0 & 0 \\
    1 & & 2 & 3 & 1 & 0 \\
    2 & & & 11 & 6 & 0 \\
    3 & & & & 18 & 0 \\
    4 & & & & & 24 \\
  \end{tabular}
  \qquad
  \begin{tabular}[t]{r|rrrrr}
    $\textsf{GPWI}(4)$\!\! & 0 & 1 & 2 & 3 & 4 \\\hline
    0 & 0 & 0 & 0 & 0 & 0 \\
    1 & & 2 & 3 & 1 & 0 \\
    2 & & & 10 & 6 & 0 \\
    3 & & & & 9 & 0 \\
    4 & & & & & 1 \\
  \end{tabular}
\end{table}

\begin{table}[tbp]
  \centering
  \caption{Numbers of Winkowski iposets on 5 points split by number of
    sources and targets.}
  \label{ta:wip5split}
  \smallskip
  \begin{tabular}[t]{r|rrrrrr}
    $\textsf{WIP}(5)$\!\! & 0 & 1 & 2 & 3 & 4 & 5 \\\hline
    0 & 0 & 0 & 0 & 0 & 0 & 0 \\
    1 & & 5 & 11 & 7 & 1 & 0 \\
    2 & & & 41 & 43 & 8 & 0 \\
    3 & & & & 81 & 36 & 0 \\
    4 & & & & & 96 & 0 \\
    5 & & & & & & 120 \\
  \end{tabular}
  \qquad
  \begin{tabular}[t]{r|rrrrrr}
    $\textsf{GPWI}(5)$\!\! & 0 & 1 & 2 & 3 & 4 & 5 \\\hline
    0 & 0 & 0 & 0 & 0 & 0 & 0 \\
    1 & & 5 & 11 & 7 & 1 & 0 \\
    2 & & & 39 & 36 & 8 & 0 \\
    3 & & & & 61 & 18 & 0 \\
    4 & & & & & 16 & 0 \\
    5 & & & & & & 1 \\
  \end{tabular}
\end{table}

\begin{table}[tbp]
  \centering
  \caption{Numbers of Winkowski iposets on 6 points split by number of
    sources and targets.}
  \label{ta:wip6split}
  \begin{tabular}[t]{r|rrrrrrr}
    $\textsf{WIP}(6)$\!\! & 0 & 1 & 2 & 3 & 4 & 5 & 6 \\\hline
    0 & 0 & 0 & 0 & 0 & 0 & 0 & 0 \\
    1 & & 16 & 47 & 47 & 15 & 1 & 0 \\
    2 & & & 200 & 285 & 135 & 10 & 0 \\
    3 & & & & 598 & 408 & 60 & 0 \\
    4 & & & & & 600 & 240 & 0 \\
    5 & & & & & & 600 & 0 \\
    6 & & & & & & & 720 \\
  \end{tabular}

  \medskip
  \begin{tabular}[t]{r|rrrrrrr}
    $\textsf{GPWI}(6)$\!\! & 0 & 1 & 2 & 3 & 4 & 5 & 6 \\\hline
    0 & 0 & 0 & 0 & 0 & 0 & 0 & 0 \\
    1 & & 16 & 47 & 46 & 15 & 1 & 0 \\
    2 & & & 190 & 238 & 102 & 10 & 0 \\
    3 & & & & 406 & 256 & 30 & 0 \\
    4 & & & & & 222 & 40 & 0 \\
    5 & & & & & & 25 & 0 \\
    6 & & & & & & & 1 \\
  \end{tabular}
\end{table}

\begin{table}[tbp]
  \centering
  \caption{Numbers of Winkowski iposets on 7 points split by number of
    sources and targets.}
  \label{ta:wip7split}
  \begin{tabular}[t]{r|rrrrrrrr}
    $\textsf{WIP}(7)$\!\! & 0 & 1 & 2 & 3 & 4 & 5 & 6 & 7 \\\hline
    0 & 0 & 0 & 0 & 0 & 0 & 0 & 0 & 0 \\
    1 & & 63 & 243 & 343 & 185 & 31 & 1 & 0 \\
    2 & & & 1203 & 2198 & 1609 & 391 & 12 & 0 \\
    3 & & & & 5323 & 5185 & 1605 & 90 & 0 \\
    4 & & & & & 6808 & 3720 & 480 & 0 \\
    5 & & & & & & 4800 & 1800 & 0 \\
    6 & & & & & & & 4320 & 0 \\
    7 & & & & & & & & 5040 \\
  \end{tabular}

  \medskip
  \begin{tabular}[t]{r|rrrrrrrr}
    $\textsf{GPWI}(7)$\!\! & 0 & 1 & 2 & 3 & 4 & 5 & 6 & 7 \\\hline
    0 & 0 & 0 & 0 & 0 & 0 & 0 & 0 & 0 \\
    1 & & 63 & 239 & 318 & 173 & 31 & 1 & 0 \\
    2 & & & 1096 & 1727 & 1129 & 260 & 12 & 0 \\
    3 & & & & 3284 & 2699 & 838 & 45 & 0 \\
    4 & & & & & 2864 & 1112 & 80 & 0 \\
    5 & & & & & & 595 & 75 & 0 \\
    6 & & & & & & & 36 & 0 \\
    7 & & & & & & & & 1 \\
  \end{tabular}
\end{table}

\begin{table}[tbp]
  \centering
  \caption{Numbers of Winkowski iposets on 8 points split by number of
    sources and targets.}
  \label{ta:wip8split}
  \begin{tabular}[t]{r|rrrrrrrrr}
    $\textsf{WIP}(8)$\!\! & 0 & 1 & 2 & 3 & 4 & 5 & 6 & 7 & 8 \\\hline
    0 & 0 & 0 & 0 & 0 & 0 & 0 & 0 & 0 & 0 \\
    1 & & 318 & 1533 & 2891 & 2319 & 707 & 63 & 1 & 0 \\
    2 & & & 8895 & 20195 & 20222 & 8333 & 1099 & 14 & 0 \\
    3 & & & & 56783 & 71835 & 37396 & 5688 & 126 & 0 \\
    4 & & & & & 112751 & 72140 & 17580 & 840 & 0 \\
    5 & & & & & & 74000 & 35400 & 4200 & 0 \\
    6 & & & & & & & 42120 & 15120 & 0 \\
    7 & & & & & & & & 35280 & 0 \\
    8 & & & & & & & & & 40320 \\
  \end{tabular}

  \medskip
  \begin{tabular}[t]{r|rrrrrrrrr}
    $\textsf{GPWI}(8)$\!\! & 0 & 1 & 2 & 3 & 4 & 5 & 6 & 7 & 8 \\\hline
    0 & 0 & 0 & 0 & 0 & 0 & 0 & 0 & 0 & 0 \\
    1 & & 313 & 1432 & 2413 & 1856 & 616 & 63 & 1 & 0 \\
    2 & & & 7402 & 13942 & 12152 & 4736 & 626 & 14 & 0 \\
    3 & & & & 29702 & 30062 & 14150 & 2433 & 63 & 0 \\
    4 & & & & & 36058 & 20366 & 4230 & 140 & 0 \\
    5 & & & & & & 13812 & 3507 & 175 & 0 \\
    6 & & & & & & & 1316 & 126 & 0 \\
    7 & & & & & & & & 49 & 0 \\
    8 & & & & & & & & & 1 \\
  \end{tabular}
\end{table}

\begin{table}[tbp]
  \centering
  \caption{Numbers of gp-Winkowski iposets on 9 points split by number of
    sources and targets.}
  \label{ta:wip9split}
  \smallskip
  \begin{tabular}[t]{r|rrrrrrrrrr}
    $\textsf{GPWI}(9)$\!\! & 0 & 1 & 2 & 3 & 4 & 5 & 6 & 7 & 8 & 9 \\\hline
    0 & 0 & 0 & 0 & 0 & 0 & 0 & 0 & 0 & 0 & 0 \\
    1 & & 1903 & 10109 & 20397 & 20173 & 9935 & 2123 & 127 & 1 & 0 \\
    2 & & & 57949 & 126041 & 135862 & 74501 & 18507 & 1456 & 16 & 0 \\
    3 & & & & 298002 & 355788 & 221772 & 65086 & 6585 & 84 & 0 \\
    4 & & & & & 478218 & 340355 & 115612 & 13988 & 224 & 0 \\
    5 & & & & & & 276343 & 108612 & 15337 & 350 & 0 \\
    6 & & & & & & & 49569 & 8960 & 336 & 0 \\
    7 & & & & & & & & 2555 & 196 & 0 \\
    8 & & & & & & & & & 64 & 0 \\
    9 & & & & & & & & & & 1 \\
  \end{tabular}
\end{table}

\end{document}